\documentclass[12pt]{amsart}

\usepackage{latexsym}
\usepackage{amssymb}
\usepackage{amsmath}

\newtheorem{theorem}{Theorem}[section]
\newtheorem{lemma}[theorem]{Lemma}
\newtheorem{proposition}[theorem]{Proposition}
\newtheorem{corollary}[theorem]{Corollary}

\theoremstyle{definition}
\newtheorem{definition}[theorem]{Definition}

\theoremstyle{remark}

\def\N{\mathbb{N}}
\def\Z{\mathbb{Z}}
\def\R{\mathbb{R}}

\def\F{\mathcal{F}}

\begin{document}

\title
{Intersections of shifted sets}

\author{Mauro Di Nasso}

\address{Dipartimento di Matematica\\
Universit\`a di Pisa, Italy}

\email{dinasso@dm.unipi.it}

\date{}

\begin{abstract}
We consider shifts of a set $A\subseteq\N$ by elements from 
another set $B\subseteq\N$, and prove intersection 
properties according to the relative asymptotic size of $A$ and $B$.
A consequence of our main theorem is the following:
If $A=\{a_n\}$ is such that 
$a_n=o(n^{k/k-1})$, then 
the $k$-recurrence set $R_k(A)=\{x\mid |A\cap(A+x)|\ge k\}$
contains the distance sets  of arbitrarily large finite sets.
\end{abstract}



\maketitle

\section{Introduction}

It is a well-know fact that if a set of natural numbers $A$
has positive upper asymptotic density, then its \emph{set of distances}
$$\Delta(A)\ =\ \{a'-a\mid a',a\in A, a'>a\}$$
meets the set of distances $\Delta(X)$ of any infinite set $X$ (see, \emph{e.g.}, \cite{be}).
In consequence, $\Delta(A)$ is \emph{syndetic}, that is there exists $k$
such that $\Delta(A)\cap I\ne\emptyset$ for every interval $I$ of length $k$.
It is a relevant theme of research in combinatorial number theory to investigate
properties of distance sets according to their ``asymptotic size"
(see, \emph{e.g.}, \cite{ru,sa,ess,behl}.)

\smallskip
The sets of distances are generalized by the \emph{$k$-recurrence sets},
namely the sets of those numbers that are the common distance of at least
$k$-many pairs:
$$R_k(A)\ =\ \{x\mid |A\cap(A+x)|\ge k\}\,.$$
Notice that $R_1(A)=\Delta(A)$.
We now further generalize this notion. 

\smallskip
Let $[A]^h=\{Z\subseteq A\mid |Z|=h\}$
denote the family of all finite subsets of $A$ of cardinality $h$,
namely the \emph{$h$-tuples} of $A$. 

\smallskip
\begin{definition}
For $k,h\in\N$ with $h>1$, the $(h,k)$-\emph{recurrence set} of $A$ is the following set
of $h$-tuples:
$$R_{k}^{h}(A)=\left\{\{t_1<\ldots<t_h\}\in[\N]^h\,\big|\,
|(A+t_1)\cap\ldots\cap(A+t_h)|\ge k\right\}.$$
\end{definition}

\smallskip
Note that a pair $\{t<t'\}\in R^2_k(A)\Leftrightarrow t'-t\in R_k(A)$, because trivially
$|(A+t)\cap(A+t')|=|A\cap(A+(t'-t))|$. 

\smallskip
For sets of natural numbers, we write $A=\{a_n\}$ to mean that 
elements $a_n$ of $A$ are arranged in increasing order. 
We adopt the usual ``little-O" notation, and for functions 
$f:\N\to\R$, we write $a_n=o(f(n))$ to mean that $\lim_{n\to\infty}a_n/f(n)=0$.

\smallskip
Our main result is the following.

\smallskip
\begin{itemize}
\item
\textbf{Theorem \ref{main}.}
\emph{Let $A=\{a_n\}$ and $B=\{b_n\}$ be infinite sets of natural numbers,
and let:\footnote
{~By \emph{limit inferior} of a double sequence 
$\langle c_{nm}\mid (n,m)\in\N\times\N\rangle$ we mean
$$\liminf_{n,m\to\infty}c_{nm}=\lim_{k\to\infty}\left(\inf_{n,m\ge k}c_{nm}\right).$$
}
$$\liminf_{n,m\to\infty}\frac{a_n+b_m}{n\sqrt[k]{m}}\ =\ l\,.$$
If $l<\frac{1}{\sqrt[k]{h-1}}$ then $R_k^h(A)\cap[B]^h\ne\emptyset$;
and if $l=0$ then $R_k^h(A)\cap[B]^h$ is infinite for all $h$.}
\end{itemize}

\smallskip
(Notice that when $k=1$, for every infinite set $A$ 
one has $R_1^h(A)\ne\emptyset$ for all $h$).
As a consequence of the theorem above, the following 
intersection property is obtained.

\smallskip
\begin{itemize}
\item
\textbf{Theorem \ref{thm2}.}
\emph{Let $k\ge 2$. If the infinite set $A=\{a_n\}$ is such that 
$a_n=o(n^{k/k-1})$ then $R_k(A)$ is a ``finitely Delta-set", 
that is $\Delta(Z)\subseteq R_k(A)$ 
for arbitrarily large finite sets $Z$.}
\end{itemize}

\smallskip
(When $k=1$, $R_1(A)=\Delta(A)$ is trivially a ``finitely Delta-set".)

\smallskip
All proofs contained in this paper have been first
obtained by working with the \emph{hyperintegers} of nonstandard analysis.
(Nonstandard integers seem to provide a convenient framework
to investigate combinatorial properties of numbers which
depend on density; see, \emph{e.g.}, \cite{ji1,ji2,square}.)
However, all used arguments in our original proof could be translated in terms
of limits of subsequences in an (almost) straightforward manner,
with the only inconvenience of a heavier notation. So, we 
eventually decided to keep to the usual language of elementary combinatorics.

\bigskip
\section{The main theorem}

The following finite combinatorial property will be instrumental for the
proof of our main result.

\smallskip
\begin{lemma}\label{lemma}
Let $A=\{a_1<\ldots<a_n\}$ and $B=\{b_1<\ldots<b_m\}$ be 
finite sets of natural numbers. For every $k$ there exists a
subset $Z\subseteq B$ such that

\begin{enumerate}
\item
$|\bigcap_{z\in Z}\left(A+z\right)|\ge k$.
\item
$|Z|\ge L\cdot\left(\frac{n\sqrt[k]{m}}{a_n+b_m}\right)^{\!k}$
where 
$L=\prod_{i=1}^{k-1}\frac{1-\frac{i}{n}}{1-\frac{i}{a_n+b_m}}$.
\end{enumerate}
\end{lemma}

\begin{proof}
For every $i\le m$, let $A_i=A+b_i$ be the shift of $A$ by $b_i$.
Notice that $|A_i|=|A|=n$ and $A_i\subseteq I=[1,a_n+b_m]$ for all $i$.
Then denote by $\vartheta_i:[\N]^k\to\{0,1\}$
the characteristic function of $[A_i]^k$, and for $H\in[\N]^k$ let
$$f(H)\ =\ \sum_{i=1}^m\vartheta_i(H).$$
Then:
$$\sum_{H\in[I]^k}\!\!\!f(H)\ =\
\sum_{i=1}^m\left(\sum_{H\in[I]^k}\vartheta_i(H)\right)\ =\
\sum_{i=1}^m|[A_i]^k|\ =\ \sum_{i=1}^\nu\binom{n}{k}\ =\ \nu\cdot\binom{n}{k}.$$
Since $|[I]^k|=\binom{a_n+b_m}{k}$, by the \emph{pigeonhole principle}
there exists $H_0\in[I]^k$ such that
\begin{eqnarray*}
f(H_0) &\ge&
\frac{\nu\cdot\binom{n}{k}}{\binom{a_n+b_m}{k}}\ =\
\nu\cdot\frac{n(n-1)(n-2)\cdots(n-(k-1))}{(a_n+b_m)(a_n+b_m-1)\cdots
(a_n+b_m-(k-1))}
\\
\nonumber
{} & = & \nu\cdot L\cdot\left(\frac{n}{a_n+b_m}\right)^{\!k}\ =\ 
L\cdot\left(\frac{n\sqrt[k]{m}}{a_n+b_m}\right)^{\!k}\,,
\end{eqnarray*}
where $L$ is the number defined in the statement of this lemma.
Now consider the set $\Gamma=\{i\in[1,m]\mid H_0\in[A_i]^k\}$.
We have that
$$|\Gamma|\ =\ \sum_{i=1}^m\vartheta_i(H_0)\ =\ f(H_0)\ \ge\
L\cdot\left(\frac{n\sqrt[k]{m}}{a_n+b_m}\right)^{\!k}.$$
Now, $H_0=\{h_1<\ldots<h_k\}\in\bigcap_{i\in\Gamma}[A_i]^k\Rightarrow
|\bigcap_{i\in\Gamma}A_i|\ge k$, and 
the set $Z=\{b_i\mid i\in\Gamma\}$ satisfies the thesis.
\end{proof}

\smallskip
We already noticed that 
$\{t<t'\}\in R^2_k(A)$ if and only if the distance
$t'-t\in R_k(A)$. More generally, one has the property:

\smallskip
\begin{proposition}\label{deltainr}
If $Z\in R_k^h(A)$ then its set of distances 
$\Delta(Z)\subseteq R_k(A)$.
\end{proposition}

\begin{proof}
Let $Z=\{z_1<\ldots<z_h\}$. By the hypothesis,
one finds at least $k$-many elements $\xi_1<\ldots<\xi_k$ in
the intersection $(A+z_1)\cap\ldots\cap(A+z_h)$.
This means that there exist elements $a_{ij}\in A$ 
for $i=1,\ldots,k$ and $j=1,\ldots,h$ such that
$$\xi_i=a_{i1}+z_1\ =\ \ldots\ =\ a_{ij}+z_j\ =\ \ldots\ =\ 
a_{ij'}+z_{j'}\ =\ \ldots\ =\ a_{ih}+z_h.$$
So, for all $1\le j<j'\le h$, we have that
$$a_{ij}=a_{ij'}+(z_{j'}-z_{j})\in A\cap(A+(z_{j'}-z_{j})).$$
Notice that $a_{ij}< a_{i'j}$ for $i<i'$,
so $A\cap(A+(z_{j'}-z_{j}))$ contains at least $k$-many elements.
We conclude that $z_{j'}-z_{j}\in R_k(A)$ 
for all $1\le j<j'\le h$, \emph{i.e.}
$\Delta(Z)\subseteq R_k(A)$.
\end{proof}

\smallskip
We remark that the implication in the above proposition
cannot be reversed when $h>2$.
\emph{E.g.}, if $A=\{1,2,3,5,8\}$ and $F=\{1,2,4\}$ then 
$|A\cap(A+1)|=|A\cap(A+2)|=|A\cap(A+3)|=2$, and
so $\Delta(F)=\{1,2,3\}\subseteq R_2(A)$. However
$F\notin R^3_2(A)$ because $(A+1)\cap(A+2)\cap(A+4)=\emptyset$.

\smallskip
We are finally ready to prove our main theorem.

\smallskip
\begin{theorem}\label{main}
Let $A=\{a_n\}$ and $B=\{b_n\}$ be infinite sets of natural numbers,
and let
$$\liminf_{n,m\to\infty}\frac{a_n+b_m}{n\sqrt[k]{m}}\ =\ l\,.$$
If $l<\frac{1}{\sqrt[k]{h-1}}$ then $R_k^h(A)\cap[B]^h\ne\emptyset$;
and if $l=0$ then $R_k^h(A)\cap[B]^h$ is infinite for all $h$.
\end{theorem}

\begin{proof}
Pick increasing functions $\sigma,\tau:\N\to\N$ such that
$$\lim_{n\to\infty}\frac{a_{\sigma(n)}+b_{\tau(n)}}{\sigma(n)\sqrt[k]{\tau(n)}}\ =\ l\,.$$
For every $n$, apply Lemma \ref{lemma} to the finite sets
$A_n=\{a_1<\ldots<a_{\sigma(n)}\}$ and $B_n=\{b_1<\ldots<b_{\tau(n)}\}$, 
and get the existence of a subset $Z_n\subseteq B_n$ such that
\smallskip
\begin{enumerate}
\item
$\left|\bigcap_{z\in Z_n}\left(A_n+z\right)\right|\ge k$.
\item
$|Z_n| \ge L_n\cdot\left(\frac{\sigma(n)\sqrt[k]{\tau(n)}}{a_{\sigma(n)}+b_{\tau(n)}}\right)^{\!k}$
where
$L_n=\prod_{i=1}^{k-1}
\frac{1-\frac{i}{\sigma(n)}}{1-\frac{i}{a_{\sigma(n)}+b_{\tau(n)}}}$.
\end{enumerate}

Since $\lim_{n\to\infty} L_n=1$, we have that
$$\liminf_{n\to\infty}|Z_n|\ \ge\
\lim_{n\to\infty}\,L_n\cdot\left(\frac{\sigma(n)\sqrt[k]{\tau(n)}}
{a_{\sigma(n)}+b_{\tau(n)}}\right)^{\!k}\ =\
1\cdot \left(\frac{1}{l}\right)^k\ >\  h-1.$$
Let $t$ be an index such that $|Z_t|>h-1$, and pick
$z_1<\ldots<z_h\in Z_t$. Then:
$$\left|\bigcap_{i=1}^h(A+z_i)\right|\ \ge\
\left|\bigcap_{i=1}^h(A_t+z_i)\right|\ \ge\
\left|\bigcap_{z\in Z_t}(A_t+z)\right|\ \ge\ k.$$
As $Z_t\subset B$, we conclude that 
$\{z_1<\ldots<z_h\}\in R_k^h(A)\cap[B]^h$.

\smallskip
Now let us turn to the case $l=0$.
Given $s>1$, pick $j\le s$ such that the set
$T_j=\{\tau(n)\mid \tau(n)\equiv j\mod s\}$ is infinite,  
let $\xi,\zeta:\N\to\N$ be the increasing functions such that
$T_j=\{\tau(\xi(n))\}=\{s\cdot\zeta(n)+j\}$, and
let $B=\{b'_n\}$ be the set where $b'_n=b_{sn+j}$. Then for every $h>1$:
$$\liminf_{n,m\to\infty}\frac{a_n+b'_m}{n\cdot\sqrt[k]{m}}\ \le\ 
\lim_{n\to\infty}\frac{a_{\sigma(\xi(n))}+b'_{\zeta(n)}}
{\sigma(\xi(n))\cdot\sqrt[k]{\zeta(n)}}\ =$$
$$\lim_{n\to\infty}\frac{a_{\sigma(\xi(n))}+b_{\tau(\xi(n))}}
{\sigma(\xi(n))\cdot\sqrt[k]{\tau(\xi(n))} }\cdot
\sqrt[k]{\frac{s\cdot\zeta(n)+j}{\zeta(n)}}\ =\ l\cdot\sqrt[k]{s}\ =\ 0\ 
<\ \frac{1}{\sqrt[k]{h-1}}\,.$$
By what already proved above, we get the existence of an $h$-tuple
$$Z\ =\ \{z_1<z_2<\ldots<z_h\}\ \subseteq\ B'$$
such that $|\bigcap_{i=1}^h(A+z_i)|\ge k$.
It is clear from the definition of $B'$ that
$\max Z\ge b'_h\ge sh+j>s$. Since $s$ can be taken arbitrarily large,
we conclude that $R_k^h(A)\cap[B]^h$ is infinite, as desired.
\end{proof}

\smallskip
\begin{corollary}
Let $A=\{a_n\}$ and $B=\{b_n\}$ be infinite sets of natural numbers.
If there exists a function $f:\N\to\R^+$ such that
$$\limsup_{n\to\infty}\frac{a_n}{n\cdot f(n)}\ <\ \infty\quad\text{and}\quad
\lim_{n\to\infty}\frac{f(b_n)}{\sqrt[k]{n}}\ =\ 0\,,$$
then $R_k^h(A)\cap[B]^h$ is infinite for all $h$.
\end{corollary}

\begin{proof}
It directly follows from Theorem \ref{main}, since
$$\liminf_{n,m\to\infty}\frac{a_n+b_m}{n\sqrt[k]{m}}\ \le\ 
\liminf_{m\to\infty}\frac{a_{b_m}+b_m}{b_m\cdot\sqrt[k]{m}}\ =\ 
\liminf_{m\to\infty}\frac{a_{b_m}}{b_m\cdot\sqrt[k]{m}}\ \le$$
$$\le\  
\limsup_{m\to\infty}\frac{a_{b_m}}{b_m\cdot f(b_m)}\cdot
\liminf_{m\to\infty}\frac{f(b_m)}{\sqrt[k]{m}}\ =\ 0\,.$$
\end{proof}

\smallskip
An an example, we now see a property that also applies to all zero 
density sets having at least the
same ``asymptotic size'' as the prime numbers.

\smallskip
\begin{corollary}
Assume that the sets $A=\{a_n\}$ and $B=\{b_n\}$ satisfy
the conditions 
$\sum_{n=1}^\infty\frac{1}{a_n}=\infty$ and
$\log b_n=o(n^\varepsilon)$ for 
all $\varepsilon>0$. Then for every $h$ and $k$, 
there exist infinitely many $h$-tuples
$\{\beta_1<\ldots<\beta_h\}\subset B$ such that
each distance $\beta_j-\beta_i$ equals the distance of
$k$-many pairs of elements of $A$.
\end{corollary}

\begin{proof}
By the hypothesis $\sum_{n=1}^\infty\frac{1}{a_n}=\infty$
it follows that $a_n=o(n\log^2n)$, 
and so the previous corollary applies with $f(n)=\log^2n$. 
Clearly, every $h$-tuple
$\{\beta_1<\ldots<\beta_h\}\in R_k^h(A)\cap[B]^h$ satisfies
the desired property.
\end{proof}

\bigskip
\section{Finitely $\Delta$-sets}

Recall that a set $A\subseteq\N$ is called a \emph{Delta-set} 
(or \emph{$\Delta$-set} for short) if 
$\Delta(X)\subseteq A$ for some infinite $X$. A basic result 
is the following: ``If $A$ has positive upper asymptotic density, then 
$\Delta(A)\cap\Delta(X)\ne\emptyset$ for all infinite sets $X$.''
(See, \emph{e.g.}, \cite{be}.) Another relevant property is
that $\Delta$-sets are \emph{partition regular}, \emph{i.e.} 
the family $\F$ of $\Delta$-sets satisfies the following property:

\smallskip
\begin{itemize}
\item
If a set $A=A_1\cup\ldots\cup A_r$ of $\mathcal{F}$ is partitioned
into finitely many pieces, then at least one of the pieces
$A_i$ belongs to $\F$.
\end{itemize}

\smallskip
To see this, let an infinite set of distances 
$\Delta(X)=C_1\cup\ldots\cup C_r$ be finitely partitioned,
and consider the partition of the pairs $[X]^2=D_1\cup\ldots\cup D_r$ where
$\{x<x'\}\in D_i\Leftrightarrow x'-x\in C_i$. By the infinite Ramsey Theorem,
there exists an infinite $Y\subseteq X$ and an index $i$ such that
$[Y]^2\subseteq D_i$, which means $\Delta(Y)\subseteq C_i$.

\smallskip
A convenient generalization of $\Delta$-sets is the following.

\smallskip
\begin{definition}
$A$ is a \emph{finitely $\Delta$-set} (or \emph{$\Delta_f$-set} for short)
if it contains the distances of finite sets of arbitrarily large size, \emph{i.e.},
if for every $k$ there exists $|X|=k$ such that $\Delta(X)\subseteq A$.
\end{definition}

\smallskip
Trivially every $\Delta$-set is a $\Delta_f$-set, but not conversely.
For example, take any sequence $\{a_n\}$ such that $a_{n+1}>a_n\cdot n$,
let $A_n=\{a_n\cdot i\mid i=1,\ldots,n\}$, and consider
the set $A=\bigcup_{n\in\N}A_n$.
Notice that for every $n$, one has $\Delta(A_n)\subseteq A_n$, and hence $A$ is 
a $\Delta_f$-set. However $A$ is not a $\Delta$-set. Indeed, assume by contradiction
that $\Delta(X)\subseteq A$ for some infinite $X=\{x_1<x_2<\ldots\}$;
then $x_2-x_1=a_k\cdot i$ for some $k$ and some $1\le i\le k$.
Pick a large enough $m$ so that $x_m>x_2+a_k\cdot k$.
Then $x_m-x_1,x_m-x_2\in\bigcup_{n>k}A_n$, and so
$x_2-x_1=(x_m-x_1)-(x_m-x_2)\ge a_{k+1}>a_k\cdot k\ge x_2-x_1$,
a contradiction.
We remark that there exist ``large'' sets that are not $\Delta_f$-sets. 
For instance, consider the set $O$ of odd numbers; it is readily seen that
$\Delta(Z)\not\subseteq O$ whenever $|Z|\ge 3$.

\smallskip
The following property suggests the notion of $\Delta_f$-set as
combinatorially suitable.

\smallskip
\begin{proposition}
The family of $\Delta_f$-sets is partition regular.
\end{proposition}

\begin{proof}
Let $A$ be a $\Delta_f$-set, and let $A=C_1\cup\ldots\cup C_r$ be a 
finite partition. Given $k$, by the finite Ramsey theorem
we can pick $n$ large enough so that every $r$-partition of the pairs
$[\{1,\ldots,n\}]^2$ admits a homogeneous set of size $k$.
Now pick a set $X=\{x_1<\ldots<x_n\}$ with $n$-many elements
such that $\Delta(X)\subseteq A$, and consider 
the partition $[\{1,\ldots,n\}]^2=D_1\cup\ldots\cup D_r$
where $\{i<j\}\in D_t\Leftrightarrow x_j-x_i\in C_t$. Then
there exists an index $t_k$ and a set $H=\{h_1<\ldots<h_k\}$ of
cardinality $k$ such that $[H]^2\subseteq D_{t_k}$.
This means that the set 
$Y=\{x_{h_1}<\ldots<x_{h_k}\}$ is such that 
$\Delta(Y)\subseteq C_{t_k}$. Since there are only finitely many
pieces $C_1,\ldots,C_r$, there exists $t$ such that $t_k=t$ for 
infinitely many $k$. In consequence, $C_t$ is a $\Delta_f$-set.
\end{proof}

\smallskip
As a straight consequence of Theorem \ref{main}, we can
give a simple sufficient condition on the ``asymptotic size" of a set $A$
that guarantees the corresponding $k$-recurrence sets
be finitely $\Delta$-sets.

\smallskip
\begin{theorem}\label{thm2}
Let $k\ge 2$ and let the infinite set $A=\{a_n\}$ be such that 
$a_n=o(n^{k/k-1})$. Then $R_k(A)$ is a $\Delta_f$-set. 
\end{theorem}

\begin{proof}
Let $B=\N$, so $b_m=m$. By taking $m=a_n$, we obtain that
$$\liminf_{n,m\to\infty}\frac{a_n+m}{n\sqrt[k]{m}}\ \le\ 
\lim_{n\to\infty}\frac{a_n+a_n}{n\sqrt[k]{a_n}}\ =\ 
\lim_{n\to\infty}\left(2^{\frac{k}{k-1}}\cdot
\frac{a_n}{n^{\frac{k}{k-1}}}\right)^{\frac{k-1}{k}}\ =\ 0\,.$$
Then Theorem \ref{main} applies, and for every $h$
we obtain the existence of a finite set $Z$ of cardinality $h$ 
such that $Z\in R_k^h(A)\cap[B]^h=R_k^h(A)$. 
But then, by Proposition \ref{deltainr}, $\Delta(Z)\subseteq R_k(A)$.
\end{proof}


\bigskip
\begin{thebibliography}{}


\bibitem{be}
V. Bergelson,
\emph{Ergodic Ramsey Theory - an update}, in
``Ergodic Theory of $\Z^d$-actions", London Math. Soc. Lecture
Notes Series \textbf{228} (1996), pp. 1--61.

\bibitem{behl}
V. Bergelson, P. Erd\"os, N. Hindman and T. \L ukzak,
\emph{Dense Difference Sets
and their Combinatorial Structure},
in ``The Mathematics of Paul Erd\"os, I" (R. Graham and J.
Ne\u{s}et\u{r}il, eds.), Springer (1997), pp. 165--175.

\bibitem{square}
M. Di Nasso, I. Goldbring, R. Jin, S. Leth, M. Lupini and K. Mahlburg,
\emph{Progress on a sumset conjecture by Erd\"os},
Canad. J. Math., to appear. 

\bibitem{ess}
P. Erd\"os, A. S\'arkozy and V.T. S\'os,
\emph{On additive properties of general sequences},
Discrete Math. \textbf{136} (1994), pp. 75--99.

\bibitem {ji1}
R. Jin, 
\emph{Sumset phenomenon}, Proc. Amer. Math. Soc. \textbf{130} (2002), pp. 855--861.

\bibitem{ji2}
R. Jin,
\emph{Pl\"unnecke's theorem for asymptotic densities}, 
Trans. Amer. Math. Soc. \textbf{363} (2011), pp 5059--5070. 

\bibitem{ru}
I.Z. Ruzsa, \emph{On difference sets},
Studia Sci. Math. Hungar. \textbf{13} (1978), pp. 319--326.

\bibitem{sa}
A. S\'ark\"ozy,
\emph{On difference sets of sequences of integers};
part I: Acta Math. Hung. \textbf{31} (1978), pp. 125--149;
part II: Ann. Univ. Sci. Budap., Sect. Math. \textbf{21} (1978), pp. 45--53;
part III: Acta Math. Hung. \textbf{31} (1978), pp. 355--386.
\end{thebibliography}
\end{document}